\documentclass{article}
\usepackage[T2A]{fontenc}
\usepackage[utf8]{inputenc}
\usepackage[english,russian]{babel}
\usepackage[tbtags]{amsmath}
\usepackage{amsfonts,amssymb,mathrsfs,amscd}
\usepackage[hyper]{msb-a}

\makeatletter
\gdef\No{{\select@language{russian}\textnumero}}
\makeatother

\JournalName{МАТЕМАТИЧЕСКИЙ СБОРНИК}
\numberwithin{equation}{section}
\theoremstyle{plain}
\newtheorem{thm}{Теорема}
\newtheorem{lem}{Лемма}
\newtheorem{prop}{Предложение}

\theoremstyle{definition}
\newtheorem{defn}{Определение}

\newtheorem{ex}{Пример}

\newtheorem{proof}{Доказательство}
\AtEndEnvironment{proof}{\hfill$\square$}

\begin{document}
\title{Признаки неразрешимости некоторых диофантовых уравнений $n$-ой степени}

\author[Eteri~Samsonadze]{Этери~Самсонадзе}
\address{Тбилисский Государственный Университет, 
пр.\,Чавчавадзе\,~1.}
\email{etsamsonadze@gmail.com}

\date{}
\udk{511.52}
\maketitle
\begin{fulltext}
\begin{abstract}
Получены признаки неразрешимости в целых неотрицательных числах диофантова уравнения $\sum_{i=1}^{m}x_i^{n}=bc^{n}$, где $n, m \geq 2$, $b, c\in \mathbb{N}$ и каноническое разложение числа $c$ состоит из степеней простых чисел $p_i$, удовлетворяющих условию $\varphi(p_i^{k_i})\mid n$ ($p_i^{k_i
}\geq 3)$ при некоторых натуральных числах $k_i$ $(i=1,2,\ldots ,l)$; $\varphi(x)$ -- функция Эйлера. Кроме того, доказано, что если $b< m< p_i^{k_i}$ $(i=1,2,\ldots ,l)$, то это уравнение не имеет решений с натуральными компонентами $x_1,x_2,\ldots ,x_m$. Также, элементарным методом, доказано, что диофантово уравнение $x_1^n+x_2^n=(p^{s} p_1^{s_1} p_2^{s_2}\ldots p_l^{s_l})^{n}$ не имеет решений с натуральными компонентами, если $n\geq 3$, $p$ --любое простое число, $p_i$ -- простое число, удовлетворяющее условию $\varphi(p_i^{k_i})\mid n$ $(p_i^{k_i}\geq 3)$ при каком-либо натуральном $k_i$; $s$, $s_i$ -- целые неотрицательные числа $(i=1,2,..,l)$.
\end{abstract}


\markright{Признаки неразрешимости диофантовых уравнений}

\section{Введение}\label{intro}

Данная работа посвящена признакам неразрешимости в целых неотрицательных числах диофантовых уравнений
\begin{equation}\label{eq-1-1}
    \sum_{i=1}^{m}x_i^{n}=b
\end{equation}
и
\begin{equation}\label{eq-1-2}
    \sum_{i=1}^{m}x_i^{n}=bc^n,
\end{equation}
где $n,m \geq 2$, $b$ -- целое неотрицательное число, $c\in \mathbb{N}$. 

Под решением диофантова уравнения $f(x_1,x_2,\ldots ,x_m)=0$ будем подразумевать целое неотрицательное решение (т. е. такое решение 
 $(x_1,x_2,\ldots ,x_m)$, что $x_1,x_2,\ldots ,x_m$ целые неотрицательные числа), а под натуральным решением -- такое $(x_1,x_2,\ldots ,x_m)$ решение, что $x_1,x_2,\ldots ,x_m\in \mathbb{N}$ . 

Через $P\left(f(x_1,x_2,\ldots ,x_m)=d\right)$ будем обозначать число  целых неотрицательных решений уравнения $f(x_1,x_2,\ldots ,x_m)=d$.

Очевидно, что если $(x_1,x_2,\ldots ,x_m)$ является решением уравнения (\ref{eq-1-1}), то $(cx_1,cx_2,\ldots ,cx_m)$, где $c\in \mathbb{N}$, будет решением уравнения (\ref{eq-1-2}). Однако, уравнение (\ref{eq-1-2}) может иметь и другие решения, которые невозможно получить таким путем. Например, $(3,4)$ является решением уравнения $x_1^2+x_2^2=5^2$, но его невозможно получить из решения уравнения $x_1^2+x_2^2=1$ умножением компонентов решения на натуральное число. Таким образом, при любом целом неотрицательном $b$ и $c\in \mathbb{N}$ имеем:
\[P\left(\sum_{i=1}^{m}x_i^n=bc^n\right)\geq P\left(\sum_{i=1}^{m}x_i^n=b\right).\]

Статья состоит из трех частей. В первой части приводятся достаточные условия, при которых число целых неотрицательных решений уравнения (\ref{eq-1-2}) равно числу целых неотрицательных решений уравнения (\ref{eq-1-1}). Для этого вводится понятие $\varphi$-делителя числа $n$ следующим образом: простое число $p$ называется $\varphi$-делителем числа $n$, если существует такое натуральное число $k$, при котором удовлетворяется условие $\varphi(p^k)\mid n$, где $p^k\geq 3$, $\varphi(x)$ -- функция Эйлера; при этом наибольшее из возможных значений $k$, удовлетворяющих этому условию, называется степенью $\varphi$-делителя (очевидно, что при этом $n$  является четным числом). Нами доказано, что если $p_i$ является $\varphi$-делителем степени $k_i$ числа $n$ и $m<p_i^{k_i}$ $(i=1,2,\ldots ,l)$, то при любом целом неотрицательном $b$ имеем
\[P\left(\sum_{i=1}^{m}x_i^n=b(d_{p_1p_2\ldots p_l})^{n}\right)=P\left(\sum_{i=1}^{m}x_i^n=b\right),\]

\noindent где через $d_{p_1p_2\ldots p_l}$ обозначено любое натуральное число, каноническое разложение которого состоит из степеней простых чисел $p_1, p_2,\ldots , p_l$.

Во второй части работы приводятся некоторые признаки неразрешимости уравнения (\ref{eq-1-1}). Используя полученные результаты, в третьей части работы приведены признаки неразрешимости уравнения (\ref{eq-1-2}). Также доказано, что если $c=d_{p_1p_2\ldots p_l}$, где $p_i$ являются $\varphi$-делителями $k_i$-той степени числа $n$ и $b<m<p_i^{k_i}$ $(i=1,2,\ldots ,l)$, то уравнение (\ref{eq-1-2}) не имеет натуральных решений. Отсюда, в частности, следует, что уравнение 
\[x_1^{2n}+x_2^{2n}=(2^{s_1} \cdot 3^{s_2})^{2n},\]
\noindent где $n\in \mathbb{N}$,  $s_1,s_2$ -- целые неотрицательные числа, не имеет натуральных решений. Также, элементарным методом, доказано, что уравнение 
\[x_1^{n}+x_2^{n}=(p^s)^n,\]
где $n\geq 3$, $p$ -- любое простое число, $s\in \mathbb{N}$, не имеет натуральных решений. Отсюда получено, что уравнение  
\[x_1^{2n}+x_2^{2n}=(d_{p,p_1,p_2,\ldots ,p_l})^{2n},\]
\noindent где $n>1$ $(n\in \mathbb{N})$, $p$ -- любое простое число, $p_i$ -- $\varphi$-делитель любой степени числа $2n$ $(i=1,2,\ldots ,l)$, не имеет натуральных решений. В частности, не имеет натуральных решений уравнение
\[x_1^{2n}+x_2^{2n}=(2^{s_1} \cdot 3^{s_2}\cdot p^s)^{2n},\]
где $n>1$ $(n\in \mathbb{N})$, $s, s_1,s_2$ -- целые неотрицательные числа,  $p$ -- любое простое число.

Все результаты получены элементарным методом. В основном применяется элементарная теория чисел, свойства функций Кронекера и Эйлера.

\section{Условия равенства чисел целых неотрицательных решений уравнений  $\sum_{i=1}^{m}x_i^{n}=bc^n$ и  $\sum_{i=1}^{m}x_i^{n}=b$}

В данном параграфе рассматриваются диофантовые уравнения \[\sum_{i=1}^{m}x_i^{n}=b\] и \[\sum_{i=1}^{m}x_i^{n}=bc^n,\] где $n, m\geq 2$, $c\in \mathbb{N}$, $b$, $x_1, x_2,\ldots ,x_m$ -- целые неотрицательные числа. Число решений уравнения \[\sum_{i=1}^{m}x_i^{n}=d\] иногда, для краткоси, будем обозначать через $P_{m}^n(d)$ (вместо $P\left(\sum_{i=1}^{m}x_i^{n}=d\right)$). Биномиальные коэффициенты будем обозначать через $C_n^j$.

\begin{lem}\label{lem-2-1}
  При любых натуральных $n, m, l\geq 2$, $s$ и любом целом неотрицательном $b$ имеем

\begin{equation}\label{eq-2-1}
    P\left(\sum_{i=1}^{m}x_i^{n}=b(l^{n})^{s}\right)=\sum_{t_1=0}^{l-1}\sum_{t_2=0}^{l-1}\ldots \sum_{t_m=0}^{l-1}P\left(\sum_{i=1}^{m}x_i^{n}+\frac{M+\sum_{i=1}^{m}t_i^n}{l^n}=b(l^n)^{s-1}\right),
\end{equation}
 где $M=\sum_{i=1}^{m}\sum_{j=1}^{n-1}C_n^jl^{n-j}x_i^{n-j}t_{i}^j$.
\end{lem}

\begin{proof}
    Очевидно, что для любой функции $f(x_1,x_2,\ldots ,x_m)$ и $d\in \mathbb{R}$ имеем:
    \begin{equation}\label{eq-2-2}
P\left(f(x_1,x_2,\ldots ,x_m)=d\right)=\sum_{x_1}\sum_{x_2}\ldots \sum_{x_m}\delta(f(x_1,x_2,\ldots ,x_m);d),
    \end{equation}

    \noindent где $\delta(x;y)$ -- функция Кронекера, а под символом $\sum_{x}$  подразумевается суммирование по всем целым неотрицательным $x$.

    Т. к. для любой функции $f(x)$, любого $d\in \mathbb{R}$ и любого натурального $l\geq 2$  имеем
    \[\sum_{x}\delta(f(x);d)=\sum_{t=o}^{l-1}\sum_{x}\delta(f(lx+t);d),\]
    то из формулы (\ref{eq-2-2}) получим, используя формулу бинома Ньютона, что 
    \vskip+3mm
 \[P\left(\sum_{i=1}^{m}x_i^{n}=b(l^{n})^{s}\right)=\sum_{t_1=0}^{l-1}\sum_{t_2=0}^{l-1}\ldots \sum_{t_m=0}^{l-1}P\left(\sum_{i=1}^{m}(lx_i+t_i)^n=b(l^n)^{s}\right)= \]
 \[=\sum_{t_1=0}^{l-1}\sum_{t_2=0}^{l-1}\ldots \sum_{t_m=0}^{l-1}P\left(\sum_{i=1}^{m}l^{n}x_i^{n}+\sum_{i=1}^{m}\sum_{j=1}^{n-1}C_{n}^{j}l^{n-j}x_i^{n-j}t_i^{j}+\sum_{i=1}^{m}t_i^{n}=b(l^{n})^{s}\right),\]

откуда следует формула (\ref{eq-2-1}). 
\end{proof}

Очевидно, что в правой части формулы (\ref{eq-2-1}) ненулевыми могут быть лишь те слагаемые, где $l^n \mid (M+\sum_{i=1}^{m}t_i^{n})$.
\vskip+2mm
Если $(x_1,x_2,\ldots ,x_m)$ является решением уравнения \[\sum_{i=1}^{m}x_i^n=b,\] то $(cx_1,cx_2,\ldots ,cx_m)$, где $c\in \mathbb{N}$, является решением уравнения \[\sum_{i=1}^{m}x_i^n=bc^n.\]
Поэтому в общем случае, при любом целом неотрицательном $b$ и $c\in \mathbb{N}$, имеем 
\[P\left(\sum_{i=1}^mx_i^n=b\right)\leq P\left(\sum_{i=1}^mx_i^n=bc^n\right).\]

Приведем достаточные условия, при которых 
\[P\left(\sum_{i=1}^mx_i^n=bc^n\right)=P\left(\sum_{i=1}^mx_i^n=b\right)\] при любом целом неотрицательном $b$.

Из формулы (\ref{eq-2-1}) следует
\begin{prop}\label{prop-1}
Пусть $l>1$ и $x^n\equiv 1(mod \;l)$ при любом целом неотрицательном $x$, не делящемся на $l$. Тогда, если $m<l$, то при любом целом неотрицательном $b$ и натуральном $s$ имеем:
\[P\left(\sum_{i=1}^{m}x_i^n=b(l^s)^n\right)=P\left(\sum_{i=1}^mx_i^n=b\right).\]
\end{prop}

\begin{proof}
  Имеем $l\mid t_i^n$ при $l\mid t_i$, и, согласно условию теоремы, $t_i^n\equiv 1 (mod \; l)$  при $l\nmid t_i$. Поэтому, если среди чисел $t_1,t_2,\ldots ,t_m$ есть число, не делящееся на $l$, то $\sum_{i=1}^{n} t_i^n\equiv d \; (mod \; l)$, где $d\in \lbrace 1,2,\ldots ,m\rbrace$. Т. к. $m<l$, то в этом случае получим
  \begin{equation}\label{eq-2-3}
l\nmid \sum_{i=1}^{n} t_i^n.
  \end{equation}
Т. к. $M=\sum_{i=1}^{n} \sum_{j=1}^{n-1}C_n^j l^{n-j}x_i^{n-j}t_i^{j} $ делится на $l$, то из (\ref{eq-2-3}) следует, что \[\frac{M+\sum_{i=1}^{n}t_i^n}{l^n} \not \in \mathbb{N}.\]

Таким образом, если среди $t_1, t_2,..., t_m$ есть число, не делящееся на $l$, то 
$$P(\sum_{i=1}^{m}x_i^n+\frac{M+\sum_{i=1}^{n}t_i^n}{l^n}=b(l^n)^{s-1})=0.$$
Из формулы (\ref{eq-2-1}) следует, что $t_i$ делится на $l$ тогда и только тогда, когда $t_i=0$, т.к.  $0\leq t_i<l$. Поэтому в правой части формулы (\ref{eq-2-1}) отличным от нуля может быть только то слагаемое, которое соответствует нулевым значениям $t_i$ $(i=1,2,...,m)$. Но если $t_i=0$ при $i=1,2,...,m$, то $\sum_{i=1}^{m}t_i^{n}=0$, $M=0$ и $\frac{M+\sum_{i=1}^{m}t_i^n}{l^n}=0$

Таким образом, из формулы (\ref{eq-2-1}) получим, что \[P_{m}^{n}(b(l^n)^s)=P_{m}^{n}(b(l^n)^{s-1})\] при любом целом неотрицательном $b$ и $s\in \mathbb{N}$. Если $s-1\in \mathbb{N}$, то аналогично получим \[P_{m}^{n}(b(l^n)^{s-1})=P_{m}^{n}(b(l^n)^{s-2}),\] 
\noindent и т. д. продолжая, получим
\[P_{m}^{n}(b(l^n)^s)=P_{m}^{n}(b).\]
\end{proof}

Из Предложения \ref{prop-1} и малой теоремы Ферма следует, что если $p$ простое число, $p\geq 3$ $s, n\in \mathbb{N}$ и $m<p$, то
\begin{equation}\label{eq-2-4}
P\left(\sum_{i=1}^{m} x_{i}^{(p-1)n}=b(p^s)^{(p-1)n}\right)=P\left(\sum_{i=1}^{m} x_{i}^{(p-1)n}=b\right).
\end{equation}
\noindent при любом целом неотрицательном $b$.

\begin{ex}\label{ex-1}
При любом целом неотрицательном $b$ и $n, s\in \mathbb{N}$ имеем 
\begin{equation}\label{eq-2-5}
    P\left(x_1^{2n}+x_2^{2n}=b(3^s)^{2n}\right)=P\left(x_1^{2n}+x_2^{2n}=b\right);
\end{equation}

\begin{equation}\label{eq-2-6}
    P\left(\sum_{i=1}^{m}x_i^{4n}=b(5^s)^{4n}\right)=P\left(\sum_{i=1}^{m}x_i^{n}=b\right)
\end{equation}

\noindent при $m\leq 4$.
\end{ex}
\vskip+2mm
\begin{lem}\label{lem-2-2}
Если  $n!=p_1^{k_1}p_2^{k_2}\ldots p_l^{k_l}$ -- каноническое разложение числа  $n!$ ( $p_i$-- простые числа, $k_i\in \mathbb{N}$), то $k_i<n$ $(i=1,2,\ldots ,l)$.
\end{lem}

\begin{proof}
Как известно \cite{B}, при $1\leq i \leq l$ имеем
\[k_i=\left[\frac{n}{p_i}\right]+\left[\frac{n}{p_i^{2}}\right]+\left[\frac{n}{p_i^{3}}\right]+\ldots +0.\] Поэтому 
\[k_i<\frac{n}{p_i}+\frac{n}{p_i^{2}}+\frac{n}{p_i^3}+\ldots =\frac{\frac{n}{p_i}}{1-\frac{1}{p_i}}\leq n.\]
\end{proof}

\begin{lem}\label{lem-2-3}
Если $p^k\mid n$ и $1\leq j\leq n-1$, где $p$ -- простое число, $k\in \mathbb{N}$, то $p^{k+1}\mid C_{n}^{j}p^{n-j}.$
\end{lem}

\begin{proof}
Т. к. $1\leq j\leq n-1$, то 
\begin{equation}\label{eq-2-7}
C_n^jp^{n-j}=\frac{(j+1)(j+2)\ldots n}{(n-j)!}p^{n-j},
\end{equation}
где $n-j\geq 1$.

Из Леммы \ref{lem-2-2} следует, что $\frac{p^{n-j}}{(n-j)!}=\frac{p^{l}}{z}$, где $l\in N$ $p\nmid z$. Т. к. $p^k\mid n$, то отсюда и из формулы (\ref{eq-2-7}) следует, что $C_n^jp^{n-j}=\frac{sp^{k+1}}{z},$ где $s\in N$, $p\nmid z$. Поэтому $p^{k+1}\mid C_n^{j}p^{n-j}$.
\end{proof}

Как известно \cite{HW}, если $p$ простое число, $k\in \mathbb{N}$, то
\begin{equation}\label{eq-2-8}
x^{\varphi(p^k)}\equiv 1\; (mod \; p^k) 
\end{equation}
при $p\nmid x$ и
\begin{equation}\label{eq-2-9}
x^{\varphi(p^k)}\equiv 0\; (mod \; p^k)
\end{equation}
при $p\mid x$.

Из формул (\ref{eq-2-8}), (\ref{eq-2-9}), Леммы \ref{lem-2-3} и формулы (\ref{eq-2-1}) (при $l=p$) следует

\begin{prop}\label{prop-2}
Если $\varphi(p^{k+1})\mid n$, где $p$ простое число, $k\in \mathbb{N}$, $m<p^{k+1}$, $s\in \mathbb{N}$, то
\[P\left(\sum_{i=1}^{m}x_i^{n}=b(p^s)^n\right)=P\left(\sum_{i=1}^{m}x_i^{n}=b\right).\]
\end{prop}

\begin{proof}
Т. к. $\varphi(p^{k+1})\mid n$, то из формул (\ref{eq-2-8}) и (\ref{eq-2-9}) следует, что
\[p^{k+1}\mid t_i^n\]
при $p\mid t_i$ и
\[t_i^n\equiv 1 \;  (mod \; p^{k+1})\] при  $p\nmid t_i$ ($i\in \lbrace 1,2,\ldots ,m\rbrace$).
Поэтому, если среди $t_1,t_2,\ldots ,t_m$ есть число, не делящееся на $p$, то 
\[\sum_{i=1}^{m}t_i^{n}\equiv r \; (mod\; p^{k+1}),\] 
где $r\in \lbrace 1,2,\ldots ,m\rbrace$.
Т. к. $m<p^{k+1}$, то отсюда следует, что
\begin{equation}\label{eq-2-10}
    p^{k+1}\nmid \sum_{i=1}^{m}t_i^{n}.
\end{equation}

Т. к. $\varphi(p^{k+1})\mid n$, то $p^k\mid n$. Поэтому из Леммы \ref{lem-2-3} следует, что число \[M=\sum_{i=1}^{m}\sum_{j=1}^{n-1}C_{n}^{j}p^{n-j}x^{n-j}t^j\]
делится на $p^{k+1}$. Поэтому из (\ref{eq-2-10}) следует, что 
\[p^{k+1}\nmid M+\sum_{i=1}^{m}t_i^n.\]
 Отсюда следует, что 
\[p^{n}\nmid M+\sum_{i=1}^{m}t_i^n,\]
т. к. $\varphi(p^{k+1})\mid n$ и $\varphi(p^{k+1})\geq k+1$ при $k\in N$ (что нетрудно доказать методом математической индукции по $k$). 

Таким образом, если среди $t_1,t_2,\ldots ,t_m$ есть число, не делящееся на $p$, то
\[P\left(\sum_{i=1}^{m}x_i^n+\frac{M+\sum_{i=1}^{m}t_i^n}{p^n}=b(p^n)^{s-t}\right)=0.\]

Из формулы (\ref{eq-2-1}) при $l=p$ следует, что $p\mid t_i$ тогда и только тогда, когда $t_i=0$, т.к. $0\leq t_i<p$. Поэтому в правой части равенства (\ref{eq-2-1}) отличным от нуля может быть только то слагаемое, которое соответствует нулевым значениям $t_i$ ($i=1,2,\ldots ,m)$. Но если $t_1=t_2=\ldots =t_m=0$, то  $\sum_{i=1}^{m}t_i^n=0$, $M=0$ и 
\[\frac{M+\sum_{i=1}^{m}t_i^n}{p^n}=0.\]

Таким образом, из формулы (\ref{eq-2-1}) получим, что \[P_{m}^n(b(p^n)^{s})=P_m^{n}(b(p^{n})^{s-1}).\] Если $s-1\in \mathbb{N}$, то аналогично получим, что \[P_{m}^n(b(p^n)^{s-1})=P_m^{n}(b(p^{n})^{s-2})\] и т. д. продолжая, получим  \[P_{m}^n(b(p^n)^{s})=P_m^{n}(b)\] при любом целом неотрицательном $b$ и $s\in \mathbb{N}$.
\end{proof}

Из Предложения \ref{prop-2} и формулы (\ref{eq-2-4}) следует
\begin{thm}\label{thm-1}
При любом целом неотрицательном $b$ имеем:
\begin{equation}\label{eq-2-11}
P\left(\sum_{i=1}^m x_i^{\varphi(p^k)l}=b(p^s)^{\varphi(p^k)l}\right)=P\left(\sum_{i=1}^m x_i^{\varphi(p^k)l}=b\right),
\end{equation}
\noindent если $m\leq p^{k}-1$, где $p$ -- простое число, $k\in \mathbb{N}$, $p^{k}\geq 3,l,s\in \mathbb{N}$.
\end{thm}
\vskip+3mm

Как уже было сказано выше, через  $d_{p_1,p_2,\ldots ,p_l}$, где  $p_1,p_2,\ldots ,p_l$ простые числа, будем обозначать любое натуральное число, каноническое разложение которого состоит из степеней простых чисел $p_1,p_2,\ldots ,p_l$, т.е. \[d_{p_1,p_2,\ldots ,p_l}=p_1^{n_1}p_2^{n_2}\ldots p_l^{n_l},\]
\noindent где $n_1,n_2,\ldots ,n_l$ -- любые натуральные числа.

\begin{lem}\label{lem-2-4}
Если $P_m^{n}(b(d_{p_{i}})^n)=P_m^{n}(b)$ при $m\leq m_i$ $(i=1,2,\ldots ,l)$, то \[P_m^{n}(b(d_{p_1p_2\ldots p_l})^{n})=P_m^{n}(b)\]
при $m\leq min(m_1,m_2,\ldots, m_l)$.
\end{lem}

\begin{proof} Имеем
\[P_m^{n}(b(d_{p_1p_2\ldots p_l})^{n})=P_m^{n}(b(p_1^{n_1}p_2^{n_2}\ldots p_{l-1}^{n_{l-1}})^{n}(p_l^{n_l})^n)\] $(n_1,n_2,\ldots ,n_l\in \mathbb{N})$. Если $m\leq m_{l}$, то из условия теоремы следует, что
\[P_m^{n}(b(p_1^{n_1}p_2^{n_2}\ldots p_{l-1}^{n_{l-1}})^{n} (p_l^{n_l})^{n})=P_m^{n}(b(p_1^{n_1}p_2^{n_2}\ldots p_{l-2}^{n_{l-2}})^{n} (p_{l-1}^{n_{l-1}})^{n}).\]

\noindent Если $m\leq m_{l-1}$, то аналогично получаем, что 
\[P_m^{n}(b(p_1^{n_1}p_2^{n_2}\ldots p_{l-2}^{n_{l-2}})^{n} (p_{l-1}^{n_{l-1}})^{n})=P_m^{n}(b(p_1^{n_1}p_2^{n_2}\ldots p_{l-3}^{n_{l-3}})^{n} (p_{l-2}^{n_{l-2}})^{n}),\]
и т. д. продолжая, получим, что
\[P_m^{n}(b(d_{p_1p_2\ldots p_l})^{n})=P_m^{n}(b)\]
\noindent при $m\leq \min(m_1,m_2,\ldots ,m_l).$
 \end{proof}

 \begin{defn}\label{defn-1}
Пусть $n$ -- четное натуральное число. Простое число $p$ будем называть $\varphi$-делителем числа $n$, если существует натуральное число $k$, при котором удовлетворяется условие $\varphi(p^{k})\mid n$ ($p^k\geq 3)$. Наибольшее из возможных значений $k$, удовлетворяющих этому условию, будем называть степенью $\varphi$-делителя $p$ числа $n$.
\end{defn}

 Из Теоремы \ref{thm-1} и Леммы \ref{lem-2-4} следует
\begin{thm} \label{thm-2}
 При любом целом неотрицательном $b$ имеем 
 \begin{equation}\label{eq-2-12}
 P\left(\sum_{k=1}^m x_i^{n}=b(d_{p_1p_2\ldots p_l})^{n}\right)=P\left(\sum_{k=1}^m x_i^{n}=b\right),
 \end{equation}
 если $p_i$ является $\varphi$-делителем $k_i$-той степени числа $n$ $(i=1,2,\ldots l)$ и $m\leq \min(p_1^{k_1}-1, p_2^{k_2}-1,\ldots , p_l^{k_l}-1$). В частности,
\begin{equation}\label{eq-2-13}
P\left(x_1^{n}+x_2^{n}=b(d_{p_1p_2\ldots p_l})^{n}\right)=P\left(x_1^{n}+x_2^{n}=b\right),
\end{equation}
\noindent если $p_i$ $(i=1,2,\ldots ,l)$ являются $\varphi$-делителями любой степени числа $n$. 
\end{thm}

\begin{ex}\label{ex-2}
Из Теорем \ref{thm-1} и \ref{thm-2} следует, что при любом целом неотрицательном $b$ и натуральных $m\geq 2, n,k,s$, имеем
\begin{equation}\label{eq-2-14}1.\; \; \; \; \;  \; P\left(\sum_{i=1}^m x_i^{2^kn}=b(2^s)^{2^kn}\right)=P\left(\sum_{i=1}^m x_i^{2^kn}=b\right)
\end{equation}
при $m\leq 2^{k+1}-1$ (т. к. $2^k=\varphi(2^{k+1})$).

В частности, 
\[P\left(\sum_{i=1}^m x_i^{2n}=b(2^s)^{2n}\right)=P\left(\sum_{i=1}^m x_i^{2n}=b\right)\]
при $m\leq 3$.
\begin{equation}\label{eq-2-15}
2. \; \; \; \; \;  \; \; P\left(\sum_{i=1}^m x_i^{2\cdot 3^kn}=b(3^s)^{2\cdot 3^kn}\right)=P\left(\sum_{i=1}^m x_i^{2\cdot 3^kn}=b\right);
\end{equation}
при $m\leq 3^{k+1}-1$ (т. к. $2\cdot 3^k=\varphi(3^{k+1})$);
\begin{equation}\label{eq-2-16}
3. \; \; \; \; \;  \; P\left(\sum_{i=1}^m x_i^{4\cdot 5^kn}=b(5^s)^{4\cdot 5^kn}\right)=P\left(\sum_{i=1}^m x_i^{4\cdot 5^kn}=b\right) 
\end{equation}
при $m\leq 5^{k+1}-1$ (т. к. $4\cdot 5^k=\varphi(5^{k+1})$).

\begin{equation}\label{eq-2-17}
4. \; \; \; \; \;  \; P\left(x_1^{2n}+x_2^{2n}=b(d_{2,3})^{2n}\right)=P\left(x_1^{2n}+x_2^{2n}=b\right)
\end{equation}
(т. к. $\varphi(2^2)\mid 2$, $\varphi(3)\mid 2$ и $\min(2^2-1,3-1)=2$).

\begin{equation}\label{eq-2-18}
5. \; \; \; \; \;  \; \;  P\left(\sum_{i=1}^m x_i^{4n}=b(d_{2,5})^{4n}\right)=P\left(\sum_{i=1}^m x_i^{4n}=b\right)
\end{equation}
при $m\leq 4$ (т. к. $\varphi(2^3)\mid 4$, $\varphi(5)\mid 4$ и $\min(2^3-1,5-1)=4$).

\begin{equation}\label{eq-2-19}
6. \; \; \; \; \;  \; \; \; \; \; \; \; P\left(x_1^{4n}+x_2^{4n}=b(d_{3,5})^{4n}\right)=P\left(x_1^{4n}+x_2^{4n}=b\right)
\end{equation}
(т.к. $\varphi(3)\mid 4$, $\varphi(5)\mid 4$ и $\min(3-1,5-1)=2)$.

\begin{equation}\label{eq-2-20}
7. \; \; \; \; \;  \; P\left(x_1^{4n}+x_2^{4n}=b(d_{2,3,5})^{4n}\right)=P\left(x_1^{4n}+x_2^{4n}=b\right)
\end{equation}
(т. к. $\varphi(2^3)\mid 4$, $\varphi(3)\mid 4$, $\varphi(5)\mid 4$ , $\min(2^3-1,3-1, 5-1)=2$).

\begin{equation}\label{eq-2-21}
8. \; \; \; \; \;  \;P\left(\sum_{i=1}^{m}x_i^{120n}=b(d_{2,3,5,7,11,13})^{120n}\right)=P\left(\sum_{i=1}^{m}x_i^{120n}=b\right)
\end{equation}
при $m\leq 6$ (т. к. $120$ делится на $\varphi(2^4)$, на $\varphi(3^2)$, на $\varphi(5^2)$, на $\varphi(7)$, на $\varphi(11)$, на $\varphi(13)$ и $\min(2^4-1,3^2-1,5^2-1,7-1,11-1,13-1)=6$).

\begin{equation}\label{eq-2-22}
9. \; \; \; \; \;  \;P\left(x_1^{10n}+x_2^{10n}=b(d_{2,3,11})^{10n}\right)=P\left(x_1^{10n}+x_2^{10n}=b\right)
\end{equation}
(т. к. $\varphi(2^2)\mid 10)$, $\varphi(3)\mid 10)$, $\varphi(11)\mid 10)$ и $\min(2^2-1,3-1,11-1)=2$).
\end{ex}

Приведенные выше теоремы можно применять для исследоваения некоторых диофантовых уравнений. Например, рассмотрим следующие уравнения:
\begin{equation}\label{eq-2-23}
1. \; x_1^6+ x_2^6 +x_3^6=233280.
\end{equation}
Т. к. $233280=5\cdot (2\cdot 3)^{6}$, $\varphi(2^2)\mid 6$, $\varphi(3^2)\mid 6$ и $3=\min(2^2-1,3^2-1)$, то из Теоремы \ref{thm-2} получим, что \[P\left(x_1^6+ x_2^6 +x_3^6=233280\right)=P\left(x_1^6+ x_2^6 +x_3^6=5\right).\]

Как нетрудно проверить, уравнение \[x_1^6+ x_2^6 +x_3^6=5\] не имеет решения. Поэтому уравнение (\ref{eq-2-23}) также не имеет решения.

\begin{equation}\label{eq-2-24}
2. \; x_1^4+ x_2^4 +\ldots +x_7^4=73728.
\end{equation}

Т. к. $73728=18\cdot (2^3)^4$ и $\varphi(2^3)\mid 4$, $7=2^3-1$, то из формулы (\ref{eq-2-14}) следует, что \[P\left(x_1^4+ x_2^4 +\ldots +x_7^4=73728\right)=P\left(x_1^4+ x_2^4 +\ldots +x_7^4=18\right).\] 

Из равенства \[x_1^4+ x_2^4 +\ldots +x_7^4=18\] следует, что $x_i\leq 2$ $(i=1,2,\ldots ,7)$. Поэтому, как нетрудно проверить, все решения последнего уравнения получаются всевозможными перестановками компонентов решения $(2,1,1,0,0,0,0)$. Поэтому все решения уравнения (\ref{eq-2-24}) получаются всевозможными перестановками компонентов в $(16,8,8,0,0,0,0)$. Из известной формулы \cite{V} для перестановок с повторениями следует, что уравнение (\ref{eq-2-24}) имеет $\frac{7!}{1!2!4!}=105$ решений. 
\vskip+2mm
Из Теоремы \ref{thm-1} также следует известный факт \cite{HW}, что числа вида $4^s(8l+7)$, где $s$ и $l$ целые неотрицательные числа, не представляются в виде $x_1^2+x_2^2+x_3^2$, где  $x_1$, $x_2$, $x_3$ целые числа. Действительно, из формулы (\ref{eq-2-14}) следует:
\begin{equation}\label{eq-2-25}
P\left(x_1^2+x_2^2+x_3^2=b(4^s)\right)=P\left(x_1^2+x_2^2+x_3^2=b\right)
\end{equation}
при любых целых неотрицатеьных $b$ и $s$. 

Принимая во внимание, что $x^2=0, 1$   или  $4\; (mod \; 8)$ при любом целом $x$, нетрудно проверить, что  $x_1^2+x_2^2+x_3^2$ несравнимо с $7$  по модулю $8$, если  $x_1$, $x_2$, $x_3$ целые числа. Поэтому из формулы (\ref{eq-2-25}) следует, что уравнение \[x_1^2+x_2^2+x_3^2=4^s(8l+7),\]
\noindent где $s$ и $l$ целые неотрицательные числа, неразрешимо.

\section{Признаки неразрешимости некоторых  диофантовых уравнений $\sum_{i=1}^{m}x_i^n=b$}

Рассмотрим диофантовое уравнение \[\sum_{i=1}^{m} x_i^n=b,\] где $n, m\geq 2$, $b, x_1, x_2, \ldots , x_m$ целые неотрицательные числа.

Очевидно, что при $0\leq b\leq m$ это уравнение разрешимо (все $m$- элементные перестановки с повторениями $(1,1,\ldots ,1,0,\ldots ,0)$, где число $1$ повторяется $b$ раз, являются его решениями). Рассмотрим случаи, когда $b>m$. 

\begin{thm}\label{thm-3}
Пусть $l\in \mathbb{N}$, $n\geq 2$. Если  $m<(1+\frac{1}{l})^n-1$ и $b\in \left[ml^n+1; (l+1)^n-1\right]$, (в частности, если $m<2^n-1$ и $b\in \left[m+1; 2^n-1\right]$), то уравнение
\begin{equation}\label{eq-3-1}
\sum_{i=1}^{m} x_i^n=b
\end{equation}
не имеет решения.
\end{thm}

\begin{proof} Очевидно, что если $0\leq x_i\leq l$ при $i=1,2,\ldots ,m$, то \[\sum_{i=1}^{m} x_i^n\leq ml^n,\] а если среди чисел $x_1, x_2, \ldots , x_m$ есть число, большее чем $l$, то \[\sum_{i=1}^{m} x_i^n\geq (l+1)^n.\] Поэтому \[\sum_{i=1}^{m} x_i^n\in \left[0, ml^n\right] \cup \left[(l+1)^n, +\infty\right).\] Вследствие этого, если \[(l+1)^n-ml^n>1,\] т. е. если  \[m<(1+\frac{1}{l})^n-\frac{1}{l^n},\] и \[b\in \left[ml^n+1; (l+1)^n-1\right]\] при  каком-либо $l\in \mathbb{N}$ (в частности, при $l=1$), то \[\sum_{i=1}^{n} x_i^n\neq b.\] 
\end{proof}

Введем обозначение: через $r_{b;c}$, где $b$ целое неотрицательное число, $c\in \mathbb{N}$, будем обозначать остаток от деления $b$ на $c$. 

\begin{lem}\label{lem-3-1}
Уравнение $\sum_{i=1}^{m} x_i^n=b$, где $b\in \mathbb{N}$, не имеет решения, если существует  натуральное число $c$, при котором удовлетворяются следующие условия:

$x^n \equiv 0$
или  $1(mod \; c)$ при любом целом неотрицательном $x$; $m<c-1$ и $r_{b;c}>m$.
\end{lem}

\begin{proof}
Из условия леммы следует, что \[\sum_{i=1}^{m} x_i^n \equiv l \; (mod \; c),\] где $0\leq l\leq m<c-1$. Поэтому \[\sum_{i=1}^{m} x_i^n\neq b,\]
если  $r_{b;c}>m$.
\end{proof}

Из Леммы \ref{lem-3-1} и формул (\ref{eq-2-8}) и (\ref{eq-2-9}) имеем 
\begin{thm}\label{thm-4}
Уравнение  
\begin{equation}\label{eq-3-2}
\sum_{i=1}^{m} x_i^{\varphi(p^k)l}=b,
\end{equation}
где $p$ простое число, $k,l\in \mathbb{N}$ ($p^k\geq 3$), $2\leq m<p^k-1$ и $r_{b;p^k}>m$, не имеет решений. 
\end{thm}

Т. к. $2=\varphi(2^2)$, то из Теоремы \ref{thm-4} следует
\begin{thm}\label{thm-5}
Уравнение
\begin{equation}\label{eq-3-3}
x_1^{2n}+x_2^{2n}=b,
\end{equation}

где $b,n\in \mathbb{N}$, не имеет решения, если $b\equiv3$ $(mod \; 4)$.
\end{thm}

Отметим, что, как известно \cite{M}, \cite{F}, уравнение 
\begin{equation}\label{eq-3-4}
    x_1^2+x_2^2=b
\end{equation}
разрешимо тогда и только тогда, когда в каноническом разложении числа $b$  нет простого числа вида $4k+3$ (где $k$ целое неотрицательное число) в нечетной степени. Наше утверждение о неразрешимости уравнения (\ref{eq-3-4}) при $b\equiv 3$ $(mod\; 4)$ следует также из этого критерия разрешимости, т.к. если каноническое разложение числа $b$ не содержит простого числа вида $4k+3$
в нечетной степени, то тогда, как нетрудно проверить, $b\not \equiv3$ $(mod\; 4)$.

Как известно \cite{HW}, при $k\geq 2$ $(k\in \mathbb{N})$ и целом неотрицательном $x$, имеем
\[x^{(2^k)}\equiv 1 \; (mod \; 2^{k+2})\] при  $2\nmid x$ и
\[x^{(2^k)}\equiv 0\; (mod \; 2^{k+2})\]  при $2\mid x$.

Поэтому из Леммы \ref{lem-3-1} следует
\begin{thm}\label{thm-6}
Уравнение 
\[\sum_{i=1}^{m}x_i^{2^kl}=b,\]
где $k\geq 2$, $b,k,l\in \mathbb{N}$, не имеет решения, если $m<2^{k+2}-1$ и $r_{b;2^{k+2}}>m$.
\end{thm}

\begin{lem}\label{lem-3-2}
Уравнение 
\[\sum_{i=1}^{m}x_i^{n}=b,\]
где $b\in \mathbb{N}$, не имеет решения, если
существует натуральное число $c$, при котором удовлетворяются следующие условия:
\vskip+2mm
$x^n\equiv 0,1$ или $-1 \; (mod \; c)$ при любом целом неотрицательном $x$; $m<\frac{c-1}{2}$ и $r_{b;c}\in \left[m+1; c-m-1\right]$.
\end{lem}
\begin{proof}
Из условия леммы следует, что \[\sum_{i=1}^{m}x_i^{n}\equiv l \;(mod \; c),\] где $l\in \lbrace-m,-m+1,-m+2,\ldots , -1,0,1,\ldots ,m\rbrace$. Поэтому, если \[\sum_{i=1}^{m}x_i^{n}=b,\] то $r_{b;c}\in \lbrace 0,1,2,\ldots ,m\rbrace\cup \lbrace c-m,c-m+1,\ldots ,c-1\rbrace$. Отсюда следует, что если $(c-m)-m>1$, т. е. если $m<\frac{c-1}{2}$, и $r_{b;c}\in \left[m+1;c-m-1\right]$, то \[\sum_{i=1}^{m}x_i^{n}\neq b.\]
\end{proof}

Как известно \cite{HW}, если  $p$-- простое число, $p\geq 3$, $k\in \mathbb{N}$, то $x^{\frac{\varphi(p^k)}{2}}\equiv 1$ или $-1\; (mod \; p^k)$ при $p\nmid x$ и $x^{\frac{\varphi(p^k)}{2}}\equiv 0\; (mod \; p^k)$ при $p\mid x$. Поэтому из Леммы \ref{lem-3-2} следует 
\begin{thm}\label{thm-7}
Уравнение 
\begin{equation}\label{eq-3-5}
\sum_{i=1}^{m}x_i^{\frac{\varphi(p^k)l}{2}}=b,
\end{equation}
где $p$ простое число, $p\geq 3$, $b, k\in \mathbb{N}$, $l$ -- нечетное число, не имеет решения, если $m<\frac{p^k-1}{2}$ и $r_{b;p^k}\in \left[m+1; p^k-m-1\right]$.
\end{thm}

\begin{ex} \; \;

1. Из Теоремы \ref{thm-6} следует, что уравнение \[\sum_{i=1}^{m}x_i^{12}=32015\] неразрешимо при  $m\leq 14$ (т.к. $2^2\mid 12$, $14<2^4-1$ и $r_{32015;2^4}=15>14$).

2. Из Теоремы \ref{thm-4} следует, что уравнение \[\sum_{i=1}^{m}x_i^6=7028\] неразрешимо при  $m\leq 7$ (т.к. $\varphi(3^2)\mid 6$, $7<3^2-1$ и $r_{7028;3^2}=8>7$).

3. Из Теоремы \ref{thm-7} следует, что уравнение \[\sum_{i=1}^{m}x_i^3=9005\] неразрешимо при  $m\leq 3$ (т.к. $3=\frac{\varphi(3^2)}{2}$, $3<\frac{9-1}{2}$ и $r_{9005;9}=5\in \left[4;5\right]$).

4. Из Теоремы \ref{thm-5} следует, что уравнение \[x_1^6+x_2^6=45399\] неразрешимо (т.к. $45399\equiv 3 $ $(mod \; 4)$).

5. Из Теоремы \ref{thm-3} следует, что уравнение \[\sum_{i=1}^{127}x_i^{15}=23607\] 
не имеет решения (т.к. $127<2^{15}-1$ и $23607\in \left[128; 2^{15}-1\right]$).

\end{ex}

\section{Признаки неразрешимости уравнения $\sum_{i=1}^{m}x_i^n=bc^n$}

Рассмотрим уравнение
\begin{equation}\label{eq-4-1}
   \sum_{i=1}^{m}x_i^n=bc^n, 
\end{equation}
где $b$ целое неотрицательное число, $n,m\geq 2$, $c\in \mathbb{N}$. Очевидно, что при $0\leq b\leq m$ это уравнение разрешимо (все $m$-элементные перестановки с повторениями $(c,c,..,c,0,0,\ldots ,0)$, где число $c$ встречается $b$ раз, являются его решениями).
\vskip+2mm
\begin{defn}\label{defn-2}
Уравнение (\ref{eq-4-1}) будем называть стандартным, если $n$ четное натуральное число и \[c=d_{p_1p_2\ldots p_l},\] где $p_i$ -- $\varphi$-делители $k_i$-той степени числа $n$ ($p_i^{k_i}\geq 3$) $(i=1,2,\ldots ,l)$ и $2\leq m\leq min (p_1^{k_1}-1, p_2^{k_2}-1,\ldots ,p_l^{k_l}-1)$.
\end{defn}
\vskip+3mm
Например, уравнение 
\[x_1^{2n}+x_2^{2n}=b(d_{p_1,p_2,\ldots ,p_l})^{2n},\]
где $n\in \mathbb{N}$, $b$ целое неотрицательное число, $p_1,p_2,\ldots ,p_l$ -- $\varphi$-делители любой степени числа $2n$, является стандартным (т. к. $m=2\leq p_i^{k_i}-1$ при $p_i^{k_i}\geq 3$ $(i=1,2,\ldots ,l)$). 
\vskip+3mm
Из Теоремы \ref{thm-2} следует
\begin{prop}\label{prop-3}
Если уравнение 
 \[\sum_{i=1}^{m}x_i^n=bc^n\] 
 стандартное, то \[P\left( \sum_{i=1}^{m}x_i^n=bc^n \right)=P\left( \sum_{i=1}^{m}x_i^n=b\right).\]
 \end{prop}
\vskip+2mm
 Отсюда следует
 \begin{thm}\label{thm-8}
 Стандартное уравнение  \[\sum_{i=1}^{m}x_i^n=bc^n\]  неразрешимо (разрешимо) тогда и только тогда, когда неразрешимо (разрешимо) уравнение  \[\sum_{i=1}^{m}x_i^n=b.\]
\end{thm}
\vskip+2mm

 Резюмируя полученные результаты (Теоремы \ref{thm-3}-\ref{thm-8}),  получим следующее утверждение.
\begin{thm}\label{thm-9}
 Пусть $p$ -- простое число, $k\in \mathbb{N}$ и уравнение  \[\sum_{i=1}^{m}x_i^n=bc^n\] стандартное. Тогда в следуюших случаях оно неразрешимо:
 \vskip+1mm
 1) $m<2^n-1$ и $b\in \left[m+1; 2^n-1\right]$;
 \vskip+1mm
 2) $\varphi(p^k)\mid n$ ($p^k\geq 3$), $m<p^k-1$ и $r_{b;p^k}>m$;
  \vskip+1mm
  3) $2^k\mid n$, где $k\geq 2$, $m<2^{k+2}-1$ и $r_{b;2^{k+2}}>m$;
  \vskip+1mm
  4) $\frac{\varphi(p^k)}{2}\mid n$, $p\geq 3$, $m<\frac{p^k-1}{2}$ и $r_{b;p^k}\in \left[m+1; p^k-m-1\right]$.
  
  \end{thm}

  \begin{ex}\label{ex-4} \; \; \;
  
  1. Уравнение 
  \begin{equation}\label{eq-4-2}
      x_1^{2n}+x_2^{2n}=b(2^{s_1}\cdot 3^{s_2})^{2n},
  \end{equation}
  где $s_1,s_2,n\in \mathbb{N}$, стандартное (т. к.  $\varphi(2^2)\mid 2n$, $\varphi(3)\mid 2n$ и $2\equiv \min(2^2-1,3-1)$). Поэтому из Теоремы \ref{thm-9} следует, что оно неразрешимо при $b\equiv 3 \; (mod \; 4)$  (т. к. $\varphi(2^2)\mid 2$) и также при $b\in \left[3; 4^n-1\right]$.

  2. Уравнение
\[\sum_{i=1}^{m}x_i^{12}=b
(d_{3,5,7})^{12},\] 
  где $m\leq 4$, стандартное (т. к.  $\varphi(3^2)\mid 12$, $\varphi(5)\mid 12$, $\varphi(7)\mid 12$ и $4= \min(3^2-1,5-1,7-1$)). Поэтому из Теоремы \ref{thm-9} следует, что это уравнение неразрешимо при $b\in \left[m+1;2^{12}-1\right]$, а также в следующих случаях: если $r_{b;3^2}>m$ (т. к. $\varphi(3^2)\mid 12$ и $4<3^2-1$); если $r_{b;7}>m$ (т. к. $\varphi(7)\mid 12$ и $4<7-1$); если $r_{b;13}>m$ (т. к. $\varphi(13)\mid 12$ и $4<13-1$); если  $r_{b;2^4}>m$ (т. к. $2^2\mid 12$ и $4<2^4-1$). 

  3. Уравнение
\[\sum_{i=1}^{m}x_i^{30}=b
(d_{2,7,11})^{30}\] 
  где $m\leq 3$, стандартное, т. к. $\varphi(2^2)\mid 30$, $\varphi(7)\mid 30$, $\varphi(11)\mid 30$ и $3=\min(2^2-1,7-1,11-1)$. Поэтому из Теоремы \ref{thm-9} следует, что оно неразрешимо при $b\in \left[m+1;2^{30}-1\right]$, и также неразрешимо в следующих случаях: если  $r_{b;3^2}>m$ (т.к. $\varphi(3^2)\mid 30$   и $3<3^2-1$); если  $r_{b;7}>m$ (т. к. $\varphi(7)\mid 30$ и $3<7-1$); если  $r_{b;11}>m$ ( т.к. $\varphi(11)\mid 30$ и $3<11-1$); если  $r_{b;5^2}\in \left[m+1;25-m-1\right]$ (т. к. $\frac{\varphi(5^2)}{2}\mid 30$ и $3<\frac{5^2-1}{2}$); если $r_{b;31}>m$ (т. к. $\varphi(31)\mid 30$  и $3<31-1$).
  \end{ex}
\vskip+3mm

Как было отмечено выше, если $(x_1,x_2,\ldots ,x_m)$ -- решение уравнения 
\begin{equation}\label{eq-4-3}
\sum_{i=1}^{m}x_i^{n}=b,
\end{equation}
то  $(cx_1,cx_2,\ldots ,cx_m)$, где $c\in \mathbb{N}$, будет решением уравнения 
\begin{equation}\label{eq-4-4}
\sum_{i=1}^{m}x_i^{n}=bc^n.
\end{equation}
Поэтому, если   
\begin{equation}\label{eq-4-5}    P\left(\sum_{i=1}^{m}x_i^{n}=bc^n\right)=P\left(\sum_{i=1}^{m}x_i^{n}=b\right)\neq 0,
\end{equation}
то все решения уравнения (\ref{eq-4-4})  получаются из решений уравнения (\ref{eq-4-3}) умножением компонентов на $c$, и наоборот, все решения уравнения (\ref{eq-4-3}) получаются из решений уравнения (\ref{eq-4-4}) делением их компонентов на $c$. 
Отсюда следует
\begin{prop}\label{prop-4}
Если \[ P\left(\sum_{i=1}^{m}x_i^{n}=bc^n\right)=P\left(\sum_{i=1}^{m}x_i^{n}=b\right),\] где $b,c\in \mathbb{N}$, то 
\[P'\left(\sum_{i=1}^{m}x_i^{n}=bc^n\right)=P'\left(\sum_{i=1}^{m}x_i^{n}=b\right),\] где через $P'\left(f(x_1,x_2,\ldots ,x_m)=0\right)$ обозначено число натуральных решений уравнения $f(x_1,x_2,\ldots ,x_m)=0$.
\end{prop}

Отсюда и из Предложения \ref{prop-3} следует
\begin{thm}\label{thm-10}
Стандартное уравнение \[\sum_{i=1}^{m}x_i^{n}=bc^n\] тогда и только тогда не имеет натуральных решений, когда уравнение \[\sum_{i=1}^{m}x_i^{n}=b\] не имеет натуральных решений. 
\end{thm}
\vskip+3mm

Т. к. при  $0\leq b<m$ уравнение \[\sum_{i=1}^{m}x_i^{n}=b\] не имеет натуральных решений, то из Теоремы \ref{thm-10} следует
\begin{thm}\label{thm-11}
Если уравнение \[\sum_{i=1}^{m}x_i^{n}=bc^n\]  стандартное, то при $0\leq b<m$ оно не имеет натуральных решений. В частности, если уравнение
\begin{equation}\label{eq-4-6}
\sum_{i=1}^{m}x_i^{n}=c^n
\end{equation}
стандартное, то оно не имеет натуральных решений.
\end{thm}

Из Теоремы \ref{thm-11} следует, что уравнение 
\begin{equation}\label{eq-4-8}
 x_1^{2n}+x_2^{2n}=(2^{s_1}\cdot 3^{s_2})^{2n},   
\end{equation}
где $n\in \mathbb{N}$, $s_1,s_2$ -- целые неотрицательные числа, не имеет натуральных решений.

\begin{thm}\label{thm-12}
Уравнение
\begin{equation}\label{eq-4-9}
 x_1^{n}+x_2^{n}=(p^s)^n,  
\end{equation}
где $n\geq 3$, $p$ -- простое число, $s\in \mathbb{N}$, не имеет натуральных решений. 
\end{thm}

\begin{proof}
Как известно \cite{BSh}, \cite{M}, \cite{P}, \cite{A}, уравнение \[x_1^{4}+x_2^{4}=z^4\] не имеет натуральных решений. Кроме того, если уравнение \[x_1^{n}+x_2^{n}=z^n\] не имеет натуральных решений, то уравнение  \[x_1^{nk}+x_2^{nk}=z^{nk}\] тоже  не будет иметь натуральных решений при $k\in \mathbb{N}$. Поэтому данную теорему достаточно доказать для нечетных $n\geq 3$. 

Допустим, что существуют $x_1,x_2\in \mathbb{N}$, удовлетворяющие  условию (\ref{eq-4-9}), где $n$  нечетное число, $n\geq 3$. Т. к.  $x_1+x_2\mid x_1^n+x_2^n$, $p$ -- простое число и  $x_1+x_2\neq 1$, то из (\ref{eq-4-9}) будет следовать, что \[x_1+x_2=p^k,\] где $k\in \mathbb{N}$. Из (\ref{eq-4-9}) следует также, что $x_1,x_2<p^s$, поэтому \[x_1+x_2<2p^s\leq p^{s+1}.\] Кроме того, т. к. из формулы бинома Ньютона следует, что  \[(x_1+x_2)^{n}>x_1^n+x_2^n,\]  то из (\ref{eq-4-9}) следует, что \[x_1+x_2>p^s.\] Таким образом, получили, что  \[p^s<x_1+x_2<p^{s+1},\] где $s\in \mathbb{N}$, откуда следует, что\[x_1+x_2\neq p^k,\] при $k\in \mathbb{N}$. Из полученного противоречия следует, что уравнение  (\ref{eq-4-9}) не имеет натуральных решений при нечетном $n\geq 3$. Таким образом, уравнение (\ref{eq-4-9}) не имеет натуральных решений при любом натуральном $n\geq 3$. 
\end{proof}

\begin{thm}\label{thm-13}
 Пусть $p$ -- любое простое число, $n>1$. Тогда уравнение 
\begin{equation}\label{eq-4-10}
    x_1^{2n}+x_2^{2n}=(p^s\cdot p_1^{s_1}p_2^{s_2}\ldots p_l^{s_l})^{2n},
\end{equation}
 где  $p_i$ -- $\varphi$-делитель любой степени числа $2n$, $s,s_i$ -- целые неотрицательные числа $(i=1,2,\ldots ,l)$, не имеет натуральных решений.   В частности, не имеет натуральных решений уравнение
\begin{equation}\label{eq-4-11}
    x_1^{2n}+x_2^{2n}=(2^{s_1}\cdot 3^{s_2}p^{s})^{2n},
\end{equation}
где  $n>1$ $(n\in \mathbb{N})$, $s_1,s_2,s$ -- целые неотрицательные числа. 
\end{thm}

\begin{proof} Из формулы (\ref{eq-2-14}) следует, что 
   \[ P\left(x_1^{2n}+x_2^{2n}=(p^s)^{2n}\cdot (p_1^{s_1}p_2^{s_2}\ldots p_l^{s_l})^{2n}\right)=P\left(x_1^{2n}+x_2^{2n}=(p^s)^{2n}\right).\]

\noindent  Поэтому из Теоремы \ref{thm-12} и Предложения \ref{prop-4} следует, что уравнения (\ref{eq-4-10}) и (\ref{eq-4-11}) не имеют натуральных решений.
\end{proof}

\begin{ex} 
Из вышеприведенных теорем и формул из Примера \ref{ex-2} следует, что следующие уравнения при $0\leq b<m$ (в частности, при $b=1$) не имеют натуральных решений ($n,s,k\in \mathbb{N},$ $m\geq 2)$:
\vskip+1mm
1. $\sum_{i=1}^{m}x_i^{2^kn}=b(2^{s})^{2^kn}$ при $m\leq 2^{k+1}-1$;
\vskip+1mm
2. $\sum_{i=1}^mx_i^{2\cdot 3^kn}=b(3^{s})^{2\cdot 3^kn}$ при $m\leq 3^{k+1}-1$;
\vskip+1mm
3. $\sum_{i=1}^m x_i^{4\cdot 5^kn}=b(5^{s})^{4\cdot 5^kn}$ при $m\leq 5^{k+1}-1$;
\vskip+1mm
4. $x_1^{2n}+x_2^{2n}=(d_{2,3})^{2n}$;
\vskip+1mm
5. $\sum_{i=1}^mx_i^{4n}=b(d_{2,5})^{4n}$ при $m\leq 4$;
\vskip+1mm
6. $\sum_{i=1}^mx_i^{120n}=b(d_{2,3,5,7.11.13})^{120n}$ при $m\leq 6$;
\vskip+1mm
7. $\sum_{i=1}^{m}x_i^{36}=b(d_{2,3,5,7,13,19,37})^{36}$ при $m\leq 4$
(т.к. $36$ делится на $\varphi(2^3)$, на $\varphi(3^3)$, на $\varphi(5)$, на  $\varphi(7)$, на  $\varphi(13)$, на $\varphi(19)$, на $\varphi(37)$,  и $4=\min(2^3-1, 3^3-1,5-1,7-1,13-1,19-1,37-1$);
\vskip+1mm
8. $x_1^{30}+x_2^{30}=(d_{2,3,5,7,11})^{30}$ (т.к. здесь, кроме числа $5$, числа $2, 3, 7, 11$  в правой части уравнения  являются $\varphi$-делителями числа $30$: $30$ делится на $\varphi(2^2)$, на $\varphi(3)$, на $\varphi(7)$, на $\varphi(11)$, но не делится на $\varphi(5^k)$, где  $k$ -- натуральное число).
\end{ex}
\end{fulltext}


\begin{thebibliography}{99}

\bibitem{BSh} 
\by З. И. Боревич, И. Р. Шафаревич 
\book Теория чисел 
\publ Наука 
\publaddr Москва 
\yr 1985

\bibitem{B} 
\by А.\,А.~Бухштаб 
\book Теория чисел 
\publ Просвещение 
\publaddr Москва 
\yr 1966

\bibitem{V}  
\by Н.\,Я.~Виленкин 
\book Комбинаторика 
\publ Наука
\publaddr Москва
\yr 1969

\bibitem{M} 
\by Ш.\,Х.~Михелович 
\book Теория чисел 
\publ Высшая школа 
\publaddr Москва
\yr 1962

\bibitem{P} 
\by М.\,М.~Постников 
\book Ввведение в теорию алгебраических чисел 
\publ Наука 
\publaddr Москва 
\yr 1982

\bibitem{A} 
\by Г.~Эдвардс 
\book Последняя теорема Ферма. Генетическое введение в алгебраическую теорию чисел 
\publ Мир 
\publaddr Москва 
\yr 1980


\bibitem{F} 
\by E.\,D.~Flath 
\book Introduction to number theory 
\publ A Wiley-Interscience publication 
\publaddr New York
\yr 1988

\bibitem{HW} 
\by G.\,H.~Hardy, E. M. Wright 
\book Introduction on the theory of numbers 
\publ Oxford University Press 
\publaddr London 
\yr 2006



\end{thebibliography}
\end{document}